\documentclass{amsart}
\usepackage{amsfonts, amsbsy, amsmath, amssymb}
\usepackage{tikz}
\usepackage{stmaryrd}
\usepackage{url}

\ifx\fiverm\undefined 
  \newfont\fiverm{cmr5} 
\fi

\newtheorem{thm}{Theorem}[section]
\newtheorem{lem}[thm]{Lemma}
\newtheorem{cor}[thm]{Corollary}

\newtheorem{conj}[thm]{Conjecture}

\newtheorem{rmk}[thm]{Remark}

\newtheorem{thm-con}[thm]{Theorem-Conjecture}
\numberwithin{equation}{section}

\theoremstyle{definition}

\allowdisplaybreaks

\newcommand{\f}{\Bbb F}
\newcommand{\fin}{F_{\text{\rm inv}}}
\newcommand{\ba}{\boldsymbol a}
\newcommand{\x}{\boldsymbol X}
\newcommand{\fbar}{\overline{\f}}
\begin{document}

\title{More on the sum-freedom of the multiplicative inverse function}

\author[Claude Carlet]{Claude Carlet}
\address{Department of Mathematics,
University Paris 8, 93526 Saint-Denis France and Department of Informatics, University of Bergen, 5005 Bergen Norway}
\email{claude.carlet@gmail.com}
\thanks{* Claude Carlet was partially supported by the Norwegian Research Council.}
\author[Xiang-dong Hou]{Xiang-dong Hou}
\address{Department of Mathematics and Statistics,
University of South Florida, Tampa, FL 33620}
\email{xhou@usf.edu}

\keywords{APN function, finite field, Lang-Weil bound,  multiplicative inverse function, sum-free function,
}

\subjclass[2020]{11G25, 11T06, 11T71, 94D10}

\begin{abstract}
In two papers entitled ``Two generalizations of almost perfect nonlinearity" and ``On the vector subspaces of $\mathbb F_{2^n}$ over which the multiplicative inverse function sums to zero", the first author has introduced and studied the notion of sum-freedom of vectorial functions, which expresses that a function sums to nonzero values over all affine subspaces of $\f_{2^n}$ of a given dimension $k\geq 2$, and he then focused on the $k$th order sum-freedom of the multiplicative inverse function $x\in \f_{2^n}\mapsto x^{2^n-2}$. Some general results were given for this function (in particular, the case of affine spaces that do not contain 0 was solved positively), and the cases of $k\in \{3,4,n-4,n-3\}$ and of $k$ not co-prime with $n$ were solved as well (negatively); but the cases of those linear subspaces of dimension $k\in \llbracket 5;n-5\rrbracket$, co-prime with $n$, were left open. The present paper is a continuation of the previous work. After studying, from two different angles, the particular case of those linear subspaces that are stable under the Frobenius automorphism, we deduce from the second approach that, for $k$ small enough (approximately, $3\le k\leq 3n/13$), the multiplicative inverse function is not $k$th order sum-free. Finally, we deduce from results previously obtained in the second paper mentioned above, that for any even $n$ and every $2\leq k\leq n-2$, the multiplicative inverse function is not $k$th order sum-free.
\end{abstract}

\maketitle

%%%%%%%%%%%%%%%%%%%%%%%%%%%%%%%%%%%%%%%%%%%
%  section 1
%%%%%%%%%%%%%%%%%%%%%%%%%%%%%%%%%%%%%%%%%%%
\section{A Brief Introduction}
Let $n$ and $k$ be positive integers such that $1\le k\le n-1$. A (so-called vectorial)  function $F:\f_{2^n}\to\f_{2^n}$ is said to be {\em $k$th order sum-free} if (see \cite{Carlet-0}), for every $k$-dimensional affine subspace $A$ of $\f_{2^n}$, we have:
\begin{equation}\label{1.1} 
\sum_{x\in A}F(x)\ne 0.
\end{equation}This notion plays a role in cryptography and presents also an interest from a geometric viewpoint. For affine geometry over finite fields and in general, see \cite[Chapter 3]{AK1} and \cite[Section 2.2]{AK2}.

Since Condition \eqref{1.1} simply corresponds for $k=1$ to the bijectivity of $F$, we shall assume $k\geq 2$. For $k=2$, it corresponds to almost perfect nonlinearity \cite{CC-Nyberg-ter} and we are interested in $k\geq 3$.

Let $\fin:\f_{2^n}\to\f_{2^n}$ be the so-called multiplicative inverse function, defined by
\[
\fin(x)=x^{2^n-2}=\begin{cases}
x^{-1}&\text{if}\ x\ne 0,\cr
0&\text{if}\ x=0.
\end{cases}
\]We continue in the present paper the work made in \cite{Carlet-pre} on the $k$th order sum-freedom of this function, which  is important for the study of finite fields, and has applications in cryptography (it is used as a substitution box in many block ciphers, in particular the currently most important one: the Advanced Encryption Standard). We know from Nyberg \cite{CC-Nyberg-ter} that $\fin$ is second order sum-free if and only if $n$ is odd.
It is also known from \cite{Carlet-0} that if $A$ is an affine subspace of $\f_{2^n}$ not containing $0$, then $\sum_{x\in A}\fin(x)\ne 0$, because 
\[
\sum_{u\in A}\frac 1u=\frac{\prod_{0\ne u\in E}u}{\prod_{u\in A}u},
\]
where $E$ is the linear subspace of $\f_{2^n}$ such that $A$ is a coset of $E$. Therefore, $\fin$ is not $k$th order sum-free if and only if there is a $k$-dimensional linear subspace $E$ of $\f_{2^n}$ such that
\begin{equation}\label{1.2}
\sum_{0\ne x\in E}\frac 1x=0.
\end{equation}

\medskip
Let us summarize the values of $k\geq 3$ for which the $k$th order sum-freedom of $\fin$ could be deduced in \cite{Carlet-0,Carlet-pre} (in some cases from more general results):

\begin{itemize}
%\item $\fin$ is $1$st order sum-free.

\item If $\text{gcd}(k,n)>1$, $\fin$ is not $k$th order sum-free,

\item $\fin$ is $k$th order sum-free if and only if it is $(n-k)$th order sum-free,

\item If $\fin$ is neither $l$th order sum-free nor $r$th order sum-free, and if $lr<n$, then $\fin$ is not $(l+r)$th order sum-free,

\item If $6\mid n$, then $\fin$ is not $k$th order sum-free for $2\le k\le n-2$,

\item For $n\ge 6$, $\fin$ is neither $3$rd order sum-free, nor $4$th order sum-free.

\end{itemize}
The above results and computer investigations suggest that the sum-freedom of $\fin$ follows a simple pattern as stated in the following conjecture.

\begin{conj}\label{Conj1.1}\cite{Carlet-pre}
For even $n$, $\fin$ is not $k$th order sum-free for $2\le k\le n-2$. For odd $n$, $\fin$ is not $k$th order sum-free for $3\le k\le n-3$.
\end{conj}

The conjecture has been confirmed for $n\le 12$ \cite{Carlet-pre} and for $6\mid n$ (as indicated above). In the present paper, we prove several new results concerning this conjecture. We find that, if $X^n+1$ has a factor of the form $X^k+a_{k-1}X^{k-1}+\cdots+a_2X^2+a_0\in\f_2[X]$, then  $\fin$ is not $k$th order sum-free. Using the Lang-Weil bound on the number of zeros of absolute irreducible polynomials over finite fields, we prove that $\fin$ is not $k$th order sum-free when $k\ge 3$ and $n\ge 13k/3+3$. We are also able to deduce that Conjecture~\ref{Conj1.1} holds for all even $n$ from a result in \cite{Carlet-pre}.

The rest of the paper is organized as follows: In Section~2, we study the companion matrix of the subspace polynomial of a $k$-dimensional subspace of $\f_{2^n}$. This leads to the conclusion that, if $X^n+1$ has a factor of the form $X^k+a_{k-1}X^{k-1}+\cdots+a_2X^2+a_0\in\f_2[X]$, then $\fin$ is not $k$th order sum-free. We then describe the values of $k$ so that such a factor exists. Section~3 provides an alternative approach to Section~2 based on the Moore determinant. Not only does this new approach lead us to the same results of Section~2, but also it prepares us for the discussion in the next section. In Section~4, using the Lang-Weil bound, we show that when $k\ge 3$ and $n\ge 13k/3+3$, $\fin$ is not $k$th order sum-free. The final section contains a short proof for Conjecture~\ref{Conj1.1} with even $n$.

%%%%%%%%%%%%%%%%%%%%%%%%%%%%%%%%%%%%%%%

\section{Case of affine subspaces (globally) invariant under the Frobenius automorphism}\label{5.3.4} 
%We know (see e.g. \cite[p. 107, Theorem 10]{CC-Sloane}) that the polynomial $x^{2^n}+x$ is the product of all monic irreducible polynomials over $\mathbb{F}_2$ with degree dividing $n$. Hence, if there exists an irreducible polynomial of the form $1+\sum_{i=2}^{k}a_ix^{2^i-1}$ whose degree $2^k-1$ divides $n$, the multiplicative inverse function over $\mathbb{F}_{2^n}$ is not $k$th-order-sum-free. REDONDANT

\subsection{A companion matrix approach}\label{spp}
We know (see e.g. \cite{CMPZ,McGS}) that a linearized polynomial $\sum_{i=0}^k b_iX^{2^{i}}\in \mathbb F_{2^n}[X]$ with $b_k=1$ has $2^k$ distinct zeros in $\mathbb{F}_{2^n}$, that is, equals a so-called subspace polynomial $L_{E_k}(X):=\prod_{u\in E_k}(X+u)$, for some $k$-dimensional linear subspace $E_k$ of $\mathbb F_{2^n}$, if and only if the so-called companion matrix 
\begin{equation}\label{defA}A= 
\left[\begin{array}{cccccc}
0&0 & 0 & \ldots & 0&b_{0}\\
 1 &0& 0 & \ldots &0& b_{1}\\
0&  1 & 0 & \ldots &0& b_{2}\\
 \vdots& \vdots&\vdots& \vdots& \vdots& \vdots\\
 0&  0 & 0 & \ldots &1& b_{k-1}\\
 \end{array}\right].
\end{equation} 
satisfies $AA^{[2]}\cdots A^{[2^{n-1}]}=I_k$, where $A^{[2^i]}$ is the matrix obtained from $A$ by applying to each of its entries
the automorphism $x\mapsto x^{2^i}$, and where $I_k$ is the $k\times k$ identity matrix. It is proved in \cite{Carlet-0,Carlet-pre} that we have $\sum_{x\in E_k}\fin(x)=0$ if and only if $b_1=0$. We are then looking at whether such matrices $A$ exist with $b_{1}=0$. \\

\noindent {\bf Remark}. Necessarily, we have $b_{0}=\prod_{u\in E_k, u\neq 0}u\neq 0$. If $\gcd(k,n)=1$, then $2^k-1$ is invertible modulo $2^n-1$, and we can assume without loss of generality that $b_{0}=1$, because by dividing each element $u$ of $E_k$ by $ b_{0}^{1/(2^k-1)}$, we change $b_{0}$ into 1. %\hfill $\diamond$\\

\medskip

The polynomial $L_{E_k}(X)$ has all its coefficients $b_i$ in $\mathbb F_2$ if and only if $L_{E_k}(X^2)=(L_{E_k}(X))^2$, that is, $E_k$ is stable under the Frobenius automorphism. Then, the condition $AA^{[2]}\cdots A^{[2^{n-1}]}=I_k$ becomes $A^n=I_k$.\\

\noindent {\bf Remark}. If $k$ divides $n$, the linearized polynomial $X^{2^k}+X$ satisfies the condition $A^n=I_k$ (as expected since $X^{2^k}+X$ is the subspace polynomial corresponding to the vector space $\mathbb F_{2^k}$), because $A$ is then the matrix of the shift over $\mathbb{F}_{2^n}^k$ and it satisfies $A^n=I_k$, thanks to the fact that $k$ divides $n$. This is also coherent with the fact that $\fin$ sums to 0 over $\mathbb{F}_{2^k}$ since it is a permutation of this field. %\hfill $\diamond$\\

\medskip

For general $k$ and $n$, recall that the Cayley-Hamilton theorem states that the characteristic equation of $A$, $\det(\lambda I-A)=0$ (where $\det$ is the determinant operation and $\lambda$ is a scalar variable), is satisfied when we replace $\lambda$ by the matrix $A$ itself (obtaining then a matrix relation), and it writes $A^k=\sum_{i=0}^{k-1}b_{i}A^i$ (see e.g. \cite[pages 146-147]{HJ} or \url{https://en.wikipedia.org/wiki/Companion_matrix}). %(see e.g. \cite{CL})
The order of the matrix $A$ equals the order of the polynomial $P(X)=\sum_{i=0}^k b_{i}X^i$ (see e.g. \cite{dar}). 
%Note that if the polynomial $x^k+\sum_{i=0}^{k-1}b_{i}x^i$ is irreducible, then we can simulate $A^i$ by $\alpha^i$, where $\alpha$ is primitive in $\mathbb{F}_{2^k}$  then we have $A^{2^k-1}=I$ and we can take $n=2^k-1$.
We then have 
%two linear endomorphisms: one is $L_{E_k}$, over $\mathbb{F}_{2^n}$, and the other, over $\mathbb{F}_2^k$, has for matrix the companion matrix $A$. These two endomorphisms have the same order, which also equals the order of the polynomial $P(x)$ (the relation between a linearized polynomial $L$ and its associated polynomial $P$ is well known, see more in \cite{WU201379}). 
\begin{thm}\label{pp9} 
If $X^n+1$ has a factor $X^k+a_{k-1}X^{k-1}+\dots +a_2X^2+a_0 \in \mathbb F_2[X]$,
then $\fin $ is not $k$th order sum-free.
\end{thm}

\begin{proof}
Let $A$ be the matrix in \eqref{defA} with 
\[
(b_0,b_1,b_2,\dots,b_{k-1},b_k)=(a_0,0,a_2,\dots,a_{k-1},1),
\]
and let $f(X)=\sum_{i=0}^kb_iX^i$. By Cayley-Hamilton, $f(A)=0$. Since $f(X)\mid X^n+1$, we have $A^n=I_k$. Therefore $f(X)=L_{E_k}(X)$ for some $k$-dimensional subspace of $\f_{2^n}$. Since $b_1=0$, we have $\sum_{x\in E_k}\fin(x)=0$.
\end{proof}

%\noindent {\bf Remark}. %JE NE COMPRENDS PLUS If $P(x)$ is irreducible then it can be used for constructing $\mathbb{F}_{2^k}$, which is a subfield of $\mathbb{F}_{2^m}$ where $m$ is the multiplicative order of 2 modulo $n$. Hence, $k$ divides $m$ (see more in \cite{BGO}). This imposes restrictions on $k$, even if $P(x)$ is allowed to be any polynomial: assume for instance that $m$ is a power of a prime $p$, then $p$ divides $k-1$ or $p$ divides $k$ (according to whether $x+1$ divides $P(x)$ or not), since $k$ is equal to a sum of divisors of $m$; if $p$ is odd (that is, at least 3), then this imposes a restriction on $k$. \\

Note that the coefficient of $X^2$ in $\sum_{i=0}^k b_iX^{2^{i}}$  equals that of $X$ in $\sum_{i=0}^k b_{i}X^i$. From Theorem~\ref{pp9}, we deduce the following corollary which gives for each (composite) $n$ a set of values of $k$ for which $\fin$ is not $k$th order sum-free.

\begin{cor}\label{c2}
Let $\prod_{i=1}^l p_i^{\alpha_i}$ be the prime factorization of $n$. For every choice of the binary hyper-matrix 
\[
\epsilon=(\epsilon_{j_1,\dots ,j_l})_{ (j_1,\dots ,j_l)\in \prod_{i=1}^l\{0,\dots ,\alpha_i\}},
\]
(with $\epsilon_{j_1,\dots ,j_l}\in \{0,1\}\subset \mathbb Z$) such that the integer \begin{equation}\label{sumepsilon}
\sum_{(j_1,\dots ,j_l)\in \{0,1\}^l}\epsilon_{j_1,\dots ,j_l}
\end{equation} 
is even (possibly 0), the multiplicative inverse function $\fin$ over $\mathbb F_{2^n}$ is not $k$th order sum-free with 
\[
k=\sum_{(j_1,\dots ,j_l)\in \prod_{i=1}^l\{0,\dots ,\alpha_i\}}\epsilon_{j_1,\dots ,j_l}\, \prod_{i\in \{1,\dots ,l\}\atop j_i\geq1}(p_i-1)\, p_i^{j_i-1}.
\]
\end{cor}

\begin{proof} 
For each value of $\epsilon$, let us consider the set that we shall denote by $D_\epsilon$, of those (distinct) divisors of $n$ equal to $\prod_{i=1}^lp_i^{j_i}$, where $\epsilon_{j_1,\dots ,j_l}=1$. We consider then the cyclotomic polynomials  (see e.g. \cite{lang}) whose indices equal these divisors. Recall that the cyclotomic polynomial of index 1 equals $X-1$ (that is, $X+1$ in characteristic 2), and the cyclotomic polynomial of index $p_i$ equals $\Phi_{p_i}(X)=(X^{p_i}-1)/(X-1)=1+X+X^2+\dots +X^{p_i-1}$. For $j\geq 1$, the cyclotomic polynomial of index $p_i^j$ equals $\Phi_{p_i^j}(X)=\Phi_{p_i}(X^{p_i^{j-1}})$ and has then degree $(p_i-1)\, p_i^{j-1}$. The cyclotomic polynomials of indices all the divisors of $n$ are obtained by iteratively applying the formula  $\Phi_{p_i^jr}(X)=\Phi_{p_ir}(X^{p_i^{j-1}})=\Phi_{r}(X^{p_i^j})/\Phi_{r}(X^{p_i^{j-1}})$, which is valid when $j\geq 1$ and $\gcd(p_i,r)=1$. The degree of $\Phi_{p_i^jr}(X)$ equals then  $(p_i^j-p_i^{j-1})\deg(\Phi_r)=(p_i-1)p_i^{j-1}\deg(\Phi_r)$.  The coefficient of $X$ in $\Phi_{p_i^jr}(X)$ equals the value at 0, for $j\geq 1$, of the polynomial derivative of the fraction $\Phi_{r}(X^{p_i^j})/\Phi_{r}(X^{p_i^{j-1}})$, that is, 
\[
\frac{\Phi_{r}(0)\Phi'_{r}(0)p_i^{j-1}0^{p_i^{j-1}-1}}{(\Phi_{r}(0))^2}, 
\]
which equals $\Phi'_{r}(0)$ if $j=1$ and 0 otherwise. Then, we have $\Phi_{p_i^jr}(X)\equiv 1+X \pmod{X^2}$ if $\Phi_r(X)\equiv 1+X \pmod{X^2}$ and $j\in \{0,1\}$, and $\Phi_{p_i^jr}(X)\equiv 1\pmod{X^2}$ otherwise. Hence, we have $\Phi_{p_1^{s_1}\cdots p_l^{s_l}}(X)\equiv 1+X \pmod{X^2}$ if $s_i\in \{0,1\}$ for all $i\in \{1,\dots ,l\}$, and $\Phi_{p_1^{s_1}\cdots p_l^{s_l}}(X)\equiv 1\pmod{X^2}$ otherwise. We have $X^n+1=\prod_{d\mid n}\Phi_d(X)$. Then, %denoting $D_\epsilon=\{\prod_{i=1}^l p_i^{j_i}; (j_1,\dots ,j_l)\in \prod_{i=1}^l\{0,\dots ,\alpha_i\}, \epsilon_{j_1,\dots ,j_l}=1\}$,
 the product $\prod_{d\in D_\epsilon}\Phi_d$ being a product of cyclotomic polynomials $\Phi_d$ such that all $d\in D_\epsilon$ are distinct, it divides $X^n+1$. We complete the proof by combining Theorem \ref{pp9} and the fact that $\sum_{u\in E,u\neq 0}1/u$ is equal to 0 if and only if the coefficient of $X^2$ in the linearized polynomial $L_E(X)=\prod_{u\in E}(X+u)$ equals 0. (Note that the condition that (\ref{sumepsilon}) is even means that the coefficient of $X$ in $\prod_{d\in D_\epsilon}(X)$ equals 0.)
\end{proof}   

Note that 
\[
\sum_{(j_1,\dots ,j_l)\in \prod_{i=1}^l\{0,\dots ,\alpha_i\}}\epsilon_{j_1,\dots ,j_l}\, \prod_{i\in \{1,\dots ,l\}\atop j_i\geq1}(p_i-1)\, p_i^{j_i-1}
\]
may be co-prime with $n$, and this corollary covers values of $k$ that are not covered by \cite{Carlet-0,Carlet-pre}. It covers in fact many values of $k$ for each $n$ (which needs to be composite, though), all the more when it has many prime divisors at large powers.\\

\noindent {\bf Example}. Take $n=12$. We have $p_1=2, p_2=3$, $\alpha_1=2, \alpha_2=1$. Since $l=2$, each $\epsilon$ is a matrix, and we shall choose $j_1$ as row-index (which makes three rows) and $j_2$ as column-index (which makes two columns). The values of the matrix $\epsilon$ satisfying the condition in Corollary \ref{c2} 
%(and that we shall always take  such that $\epsilon_{0,\dots ,0}=0$ since taking $\epsilon_{0,\dots ,0}=1$ gives monomial 1) (I don't understand this statement. XD) 
are displayed below, as the first term of each triple. The corresponding set $D_\epsilon$ of distinct divisors of $n$ (that we list in the order obtained by visiting each position equal to 1 in the first column, and then in the second column) is the second term, and the corresponding value of $k$ is the third term.\\ $$(\epsilon,D_\epsilon,k)=$$$$\left(\left[\begin{array}{l}10\\10\\00\end{array}\right],\{1,2\},2\right),\left(\left[\begin{array}{l}11\\00\\00\end{array}\right],\{1,3\},3\right),\left(\left[\begin{array}{l}10\\01\\00\end{array}\right],\{1,6\},3\right),$$$$\left(\left[\begin{array}{l}01\\10\\00\end{array}\right],\{2,3\},3\right),\left(\left[\begin{array}{l}01\\01\\00\end{array}\right],\{3,6\},4\right),\left(\left[\begin{array}{l}00\\11\\00\end{array}\right],\{2,6\},3\right),$$$$\left(\left[\begin{array}{l}01\\10\\10\end{array}\right],\{2,4,3\},5\right),\left(\left[\begin{array}{l}01\\01\\10\end{array}\right],\{4,3,6\},6\right),\left(\left[\begin{array}{l}00\\11\\10\end{array}\right],\{2,4,6\},5\right),$$$$\left(\left[\begin{array}{l}10\\10\\10\end{array}\right],\{1,2,4\},4\right),\left(\left[\begin{array}{l}11\\00\\10\end{array}\right],\{1,4,3\},5\right),\left(\left[\begin{array}{l}10\\01\\10\end{array}\right],\{1,4,6\},5\right),$$$$\left(\left[\begin{array}{l}01\\10\\01\end{array}\right],\{2,3,12\},7\right),\left(\left[\begin{array}{l}01\\01\\01\end{array}\right],\{3,6,12\},8\right),\left(\left[\begin{array}{l}00\\11\\01\end{array}\right],\{2,6,12\},7\right),$$$$\left(\left[\begin{array}{l}10\\10\\01\end{array}\right],\{1,2,12\},6\right),\left(\left[\begin{array}{l}11\\00\\01\end{array}\right],\{1,3,12\},7\right),\left(\left[\begin{array}{l}10\\01\\01\end{array}\right],\{1,6,12\},7\right),$$$$\left(\left[\!\begin{array}{l}01\\10\\11\end{array}\!\right],\{2,4,3,12\},9\right),\left(\left[\!\begin{array}{l}01\\01\\11\end{array}\!\right],\{4,3,6,12\},10\right),\left(\left[\!\begin{array}{l}00\\11\\11\end{array}\!\right],\{2,4,6,12\},9\right),$$$$\left(\left[\begin{array}{l}10\\10\\11\end{array}\right],\{1,2,4,12\},8\right),\left(\left[\begin{array}{l}11\\00\\11\end{array}\right],\{1,4,3,12\},9\right),\left(\left[\begin{array}{l}10\\01\\11\end{array}\right],\{1,4,6,12\},9\right).$$
Hence, Corollary \ref{c2} implies that the multiplicative inverse function over $\mathbb F_{2^{12}}$ is not $k$th order sum-free where $k\in \{2,3,4,5,6,7,8,9,10\}$. %\hfill $\diamond$\\

\bigskip

\noindent {\bf Other examples}.\\- For $n=6$, we have $p_1=2, p_2=3$, $\alpha_1=1, \alpha_2=1$, and
$$(\epsilon,D_\epsilon,k)=$$$$\left(\left[\begin{array}{l}01\\10\end{array}\right],\{2,3\},3\right),\left(\left[\begin{array}{l}01\\01\end{array}\right],\{3,6\},4\right),\left(\left[\begin{array}{l}00\\11\end{array}\right],\{2,6\},3\right),$$$$\left(\left[\begin{array}{l}10\\10\end{array}\right],\{1,2\},2\right),\left(\left[\begin{array}{l}11\\00\end{array}\right],\{1,3\},3\right),\left(\left[\begin{array}{l}10\\01\end{array}\right],\{1,6\},3\right);$$
- for $n=8$, $$(\epsilon,D_\epsilon,k)=$$$$\left(\left[\begin{array}{l}0\\0\\1\\0\end{array}\right],\{4\},2\right),\left(\left[\begin{array}{l}1\\1\\0\\0\end{array}\right],\{1,2\},2\right),\left(\left[\begin{array}{l}1\\1\\1\\0\end{array}\right],\{1,2,4\},4\right),$$$$\left(\left[\begin{array}{l}0\\0\\1\\1\end{array}\right],\{4,8\},6\right),\left(\left[\begin{array}{l}1\\1\\0\\1\end{array}\right],\{1,2,8\},6\right),\left(\left[\begin{array}{l}1\\1\\1\\1\end{array}\right],\{1,2,4,8\},8\right);$$
- for $n=9$, $$(\epsilon,D_\epsilon,k)=\left(\left[\begin{array}{l}0\\0\\1\end{array}\right],\{9\},6\right),\left(\left[\begin{array}{l}1\\1\\0\end{array}\right],\{1,3\},3\right),\left(\left[\begin{array}{l}1\\1\\1\end{array}\right],\{1,3,9\},9\right);$$
- for $n=10$, we have $p_1=2, p_2=5$, $\alpha_1=1, \alpha_2=1$, and
$$(\epsilon,D_\epsilon,k)=$$$$\left(\left[\begin{array}{l}01\\10\end{array}\right],\{2,5\},5\right),\left(\left[\begin{array}{l}01\\01\end{array}\right],\{5,10\},8\right),\left(\left[\begin{array}{l}00\\11\end{array}\right],\{2,10\},5\right),$$$$\left(\left[\begin{array}{l}10\\10\end{array}\right],\{1,2\},2\right),\left(\left[\begin{array}{l}11\\00\end{array}\right],\{1,5\},5\right),\left(\left[\begin{array}{l}10\\01\end{array}\right],\{1,10\},5\right).$$ 
Similarly, for $n=15$, the set of values of $k$ given by Corollary \ref{c2} is $\{3,4,\dots,12,15\}$, and for $n=21$, it is $\{3,4,\dots,18,21\}$. %\hfill $\diamond$\\

%\begin{cor}Let $C$ be any binary cyclic $[n,k]$ code such that . The multiplicative inverse function is not $(n-k)$th-order-sum-free. PROBLEME COMMENT CARACTERISER QUE LE COEF DE x EST NUL?\end{cor}{\em Proof}. We can take for $p(x)$ the generator polynomial of $C$, which has degree $n-k$ and whose coefficient of $x$ equals 0. Theorem \ref{pp9} completes the proof.\hfill $\Box$\\

\bigskip

The following corollary gives an infinite class of values of $(k,n)$ such that the multiplicative inverse $(n,n)$-function is not $k$th order sum-free.

\begin{cor}\label{cc3}
If $n$ is divisible by an integer $s\geq 2$ and if $r\leq n/s$ is the degree of any divisor of $X^{n/s}+1$ in $\mathbb F_2[X]$, then  the multiplicative inverse function is not $(sr)$th order sum-free.
\end{cor}

\begin{proof} 
Let $R(X)$ be such a divisor of degree $r$ of $X^{n/s}+1$, then $P(X)=R(X^s)$ is a divisor of degree $sr$ of $X^n+1$ and it has no term in $X$.
\end{proof}

In Corollary~\ref{cc3}, the larger the $s$, the smaller the number of the values reached by $k=sr$.\\

\noindent {\bf Remark}. Here also we can take for $R(X)$ the product of the cyclotomic polynomials of any distinct indices dividing $n/s$. The situation is simpler than in Corollary \ref{c2}, since there is no condition on $\epsilon$. But the number of values reached by $k$ is smaller. Taking $n=12$ or $n=8$ does not add new values of $k$ to those found in Corollary \ref{c2}, but for $n=6$, we obtain $k=2,4,6$ and 6 is new. %\hfill $\diamond$\\

\bigskip

\noindent {\bf Case $\boldsymbol n$ odd.} For $n$ odd, the divisors $P(X)\in \mathbb F_2[X]$ of $X^n+1$ are the generator polynomials of the binary cyclic codes of length $n$ over $\mathbb{F}_2$ \cite{CC-Sloane}. %It equals the product of distinct cyclotomic polynomials (equal to the minimal polynomials of powers of a primitive $n$th root of unity in $\mathbb{F}_{2^m}$, where $m$ is the multiplicative order of $n$ modulo 2, see e.g. \cite{CC-Sloane}). 
Given any binary cyclic code having for nonzeros 1 and at least another element, its generator polynomial $g(X)$ satisfies $X^n+1=(X+1)g(X)h(X)$ for some binary (non-trivial) polynomial $h(X)$, and one of the two polynomials $g(X)$ and $h(X)$ has no term in $X$, because $g(X)h(X)=(X^n+1)/(X+1)=\sum_{i=0}^{n-1}X^i$, and the sum of the coefficients of $X$ in $g(X)$ and $h(X)$ equals then 1. If $n$ is a prime, then the degree $k$ of this polynomial is co-prime with $n$. We do not know in general whether $k$ equals the degree of $g$ or that of $h$; if the code is the binary quadratic residue code, with $n\equiv \pm 1\pmod 8$, then these two polynomials having the same degree, we have $k=(n-1)/2$. But  there are values of $n$ for which the method does not work, because $(X^n+1)/(X+1)$ is irreducible over $\mathbb{F}_2$; this happens if and only if the cyclotomic class of 2 modulo $n$ containing 1 equals the whole $(\mathbb{Z}/n\mathbb{Z})\setminus \{0\}$, that is, 2 is a primitive element modulo $n$. %As far as we know, the structure of the binary divisors of $x^n+1$ is not well known (results exist in \cite{LT} on their degrees and in \cite{BGO} on the irreducible factors of $x^n-1$,  but they do not seem exploitable here since the degree of these factors is supposed to divide $n$). \\

\subsection{Why it is not enough to consider binary polynomials $\boldsymbol{L_{E_k}}$, that is, binary matrices $\boldsymbol A$ only}
For fixed $k$, there is a finite number (namely,  $2^{k-2}$) of binary companion $k\times k$ matrices $A$ of the form (\ref{defA}) such that $b_0=1$ and $b_1=0$, and taking for $n$ a prime number strictly larger than all prime numbers dividing the orders (necessarily larger than 1) of these matrices, we see that, for every $k$, there are values of $n$  such that $A^n\neq I_k$ for every such matrix.

 In \cite{McGM}, a particular type of trinomials of the form $X^{q^d}+bX^q+aX\in \mathbb{F}_{q^m}[X]$ was studied, where $q$ is a power of a prime\footnote{The conditions so that they split over $\mathbb{F}_{2^n}$ are strong and this means that almost all of such polynomials actually do not split.}. However, 
\\- if we take $q=2$ (and $m=n$), then since $b$ needs to be zero, being then the coefficient of $X^2$, and the equation $x^{q^d}+ax=0$ splitting only if $d$ divides $n$, we are back to the case where $k=d$ divides $n$;%; however, taking $q=2$, the equation $\sum_{i=2}^{d-1}x^{2^i}+x=0$ implies $x^{2^d}+x^2+x=0$ and we know from \cite{McGM} that if $n=(d-1)d+1$ and if $d-1$ is a power of 2, then this later equation splits over $\mathbb{F}_{2^n}$,
\\- if we assume that $m$ is a strict divisor of $n$ and  $q=2^{n/m}$, then $k=nd/m$ satisfies $\gcd(k,n)\geq n/m\geq 2$ and we get no new case where the inverse function is not $k$th order sum-free either. 

Note that when the number of cyclotomic classes (and hence, the maximal number of the minimal polynomials which are factors of $L_{E_k}(X)$) is as small as 2 (this happens with some primes $n=3,5,11,13,19,29,37,53,59,61,\dots $), the only factors with binary coefficients of $X^{2^n}+X$ are $X+1$ and $\sum_{i=0}^{n-1}X^{2^i}$ and none has a coefficient of $X^2$ equal to 0. 

With the observations above, we see that the question of determining whether the multiplicative inverse function is $k$th order sum-free over $\mathbb{F}_{2^n}$ for some $n$ and some $k$ not dividing $n$ is difficult, unless $k$ is small or large.

%%%%%%%%%%%%%%%%%%%%%%%%%%%%%%%%%%%%%%%%%%%%

\section{An Alternative Approach}

In this section, we revisit, from the viewpoint of determinants, a result from \cite[Corollary 2]{Carlet-0} and from Theorem \ref{pp9} above. This new approach will also allow us to prove in the next section that $\fin$ is not $k$th order sum-free when
$k$ is small or large (approximately $k\leq n/10$ or $k\geq 9n/10$).

\subsection{An approach through determinants}
Define
\begin{equation}\label{1.3}
\Delta(X_1,\dots,X_k)=\left|
\begin{matrix}
X_1&\cdots&X_k\cr
X_1^2&\cdots&X_k^2\cr
\vdots&&\vdots\cr
X_1^{2^{k-1}}&\cdots&X_k^{2^{k-1}}
\end{matrix}\right|,
\end{equation}
and for $0\le i\le k$,
\begin{equation}\label{1.4}
\Delta_i(X_1,\dots,X_k)=\left|
\begin{matrix}
X_1&\cdots&X_k\cr
\vdots&&\vdots\cr
X_1^{2^{i-1}}&\cdots&X_k^{2^{i-1}}\cr
X_1^{2^{i+1}}&\cdots&X_k^{2^{i+1}}\cr
\vdots&&\vdots\cr
X_1^{2^{k}}&\cdots&X_k^{2^{k}}
\end{matrix}\right|.
\end{equation}
These are polynomials in $\Bbb Z[X_1,\dots,X_k]$. However, for our purpose, we treat them as polynomials in $\f_2[X_1,\dots,X_k]$; $\Delta(X_1,\dots,X_k)$ is known as the {\em Moore determinant} over $\f_2$ \cite{Moore-BAMS-1896}.
By \cite[Lemma~3.51]{Lidl-Niederreiter-FF-1997}, 
\begin{equation}\label{1.5}
\Delta(X_1,\dots,X_k)=\prod_{{\bf 0}\ne(a_1,\dots,a_k)\in\f_2^k}(a_1X_1+\cdots+a_kX_k)=\prod_{{\bf 0}\ne\ba\in\f_2^k}(\ba\cdot\x),
\end{equation}
where ${\bf 0}=(0,\dots ,0)$, $\x=(X_1,\dots,X_k)$ and $\ba\cdot\x=a_1X_1+\cdots+a_kX_k$ for $\ba=(a_1,\dots,a_k)\in\f_2^k$. Obviously, $\Delta_0(X_1,\dots,X_k)=\Delta(X_1,\dots,X_k)^2$ and $\Delta_k(X_1,\dots,$ $X_k)=\Delta(X_1,\dots,X_k)$. However, for $1\le i\le k-1$, the formula for $\Delta_i(X_1,\dots,X_k)$ is too complicated to be useful; see \cite[Appendix]{Hou-Sze-LMA02022}. We also know that \cite[Exercise~2.15]{Hou-ams-gsm-2018}
\begin{align*}
\prod_{\ba\in\f_2^k}(Y+\ba\cdot\x)\,&=\frac{\Delta(Y,X_1,\dots,X_k)}{\Delta(X_1,\dots,X_k)}\cr
&=\frac 1{\Delta(X_1,\dots,X_k)}\sum_{i=0}^k\Delta_i(X_1,\dots,X_k)Y^{2^i}\cr
&=\sum_{i=0}^kb_{ki}Y^{2^i},
\end{align*}
where 
\[
b_{ki}=\frac{\Delta_i(X_1,\dots,X_k)}{\Delta(X_1,\dots,X_k)}.
\]
Let $v_1,\dots,v_k\in\f_{2^n}$ be linearly independent over $\f_2$. Then by \cite{Carlet-0,Carlet-pre},
\begin{equation}\label{1.6}
\sum_{0\ne x\in\langle v_1,\dots,v_k\rangle}\frac 1x=\frac{b_{k1}}{b_{k0}}=\frac{\Delta_1(v_1,\dots,v_k)}{\Delta(v_1,\dots,v_k)^2}.
\end{equation}
Therefore, $\fin$ is not $k$th order sum-free if and only if there exist $v_1,\dots,v_k\in\f_{2^n}$ such that $\Delta_1(v_1,\dots,v_k)=0$ but $\Delta(v_1,\dots,v_k)\ne 0$.

\medskip

The next theorem is equivalent to Theorem \ref{pp9}. We state and prove it for clarity.
\begin{thm}\label{T2.1}
The following two statements are equivalent:
\begin{itemize}
\item[(i)] There exists $x\in\f_{2^n}$ such that $\Delta(x,x^2,x^{2^2},\dots,x^{2^{k-1}})\ne0$ and\\ $\Delta_1(x,x^2,x^{2^2},\dots,x^{2^{k-1}})=0$.

\item[(ii)] $X^n-1$ has a factor $X^k+a_{k-1}X^{k-1}+\cdots+a_2X^2+a_0\in\f_2[X]$.
\end{itemize}
\end{thm}

\begin{proof}
Let $\sigma$ denote the Frobenius automorphism  of $\f_{2^n}$ over $\f_2$. Note that 
\begin{align*}
&\Delta(x,x^2,x^{2^2},\dots,x^{2^{k-1}})\ne0\cr
\Leftrightarrow\ &x,\sigma(x),\dots,\sigma^{k-1}(x)\ \text{are linearly independent over}\ \f_2,
\end{align*}
and 
\begin{align*}
&\Delta_1(x,x^2,x^{2^2},\dots,x^{2^{k-1}})=0\cr
\Leftrightarrow\ &x,\sigma^2(x),\dots,\sigma^{k}(x)\ \text{are linearly dependent over}\ \f_2.
\end{align*}

\medskip
(ii) $\Rightarrow$ (i). Let $\alpha\in\f_{2^n}$ be a normal element over $\f_2$. Write $X^n-1=fg$, where $f=X^k+a_{k-1}X^{k-1}+\cdots+a_2X^2+a_0$. Let $x=(g(\sigma))(\alpha)$. For each $0\ne h\in\f_2[X]$ with $\deg h<k$, $hg\not\equiv 0\pmod{X^n-1}$, whence $(h(\sigma))(x)=((hg)(\sigma))(\alpha)\ne 0$. Hence $x,\sigma(x),\dots,\sigma^{k-1}(x)$ are linearly independent over $\f_2$. 

Since $(f(\sigma))(x)=((fg)(\sigma))(\alpha)=0$, the elements $x,\sigma^2(x),\dots,\sigma^k(x)$ are linearly dependent over $\f_2$.

\medskip
(i) $\Rightarrow$ (ii). Since $x,\sigma^2(x),\dots,\sigma^k(x)$ are linearly dependent over $\f_2$, there exists $0\ne f=a_kX^k+a_{k-1}X^{k-1}+\cdots+a_2X^2+a_0\in\f_2[X]$ such that $(f(\sigma))(x)=0$. We claim that $a_k\ne0$ and $f\mid X^n-1$. Otherwise, $f_1:=\text{gcd}(f,X^n-1)$ has degree $<k$ and $(f_1(\sigma))(x)=0$. Then $x,\sigma(x),\dots,\sigma^{k-1}(x)$ are linearly dependent over $\f_2$, which is a contradiction. 
\end{proof}

\begin{cor}\label{C2.2}
If $X^n-1$ has a factor $X^k+a_{k-1}X^{k-1}+\cdots+a_2X^2+a_0\in\f_2[X]$, then $\fin$ is not $k$th order sum-free.
\end{cor}

\medskip
\noindent{\bf Note.} If we replace $k$ by $n-k$ in Corollary~\ref{C2.2}, we do not get anything new.

\subsection{Factorization of $\boldsymbol{X^n-1}$ over $\boldsymbol{\f_2}$}

This is a well-studied topic, which we briefly revisit because of Corollary~\ref{C2.2}. We are interested in the factors of $X^n-1$ of the form $X^k+a_{k-1}X^{k-1}+\cdots+a_2X^2+a_0$, where $a_i\in \mathbb F_2$, or equivalently, by considering the reciprocals, those of the form $X^k+a_{k-2}X^{k-2}+\cdots+a_0$. 

Let $n=2^et$, where $e\ge 0$ and $2\nmid t$, so that $X^n-1=(X^t-1)^{2^e}$. 
Recall that $X^t-1=\prod_{d\mid t}\Phi_d(X)$, where $\Phi_d$ is the cyclotomic polynomial of index $d$. For each $d\mid t$, let $o_d(2)$ denote the order of $2$ in $(\Bbb Z/d\Bbb Z)^\times$. The irreducible factors of $\Phi_d(X)$ in $\f_2[X]$ correspond to the $2$-cyclotomic cosets in $(\Bbb Z/d\Bbb Z)^\times$ and their degrees equal the sizes of the 
$2$-cyclotomic cosets. All $2$-cyclotomic cosets in $(\Bbb Z/d\Bbb Z)^\times$ have size $o_d(2)$ and there are $\phi(d)/o_d(2)$ such cyclotomic cosets in $(\Bbb Z/d\Bbb Z)^\times$, where $\phi$ is Euler's totient function. Hence, the multiset of the degrees of the irreducible factors of $X^t-1$ over $\f_2$ consists of $o_d(2)$ with multiplicity $\phi(d)/o_d(2)$ for all $d\mid t$. 

For $d\mid t$, let $l=o_d(2)$, we have
\begin{align*}
N_d\,&:=|\{f=X^l+b_{l-2}X^{l-2}+\cdots+b_0\in\f_2[X]\ \text{irreducible}:f\mid X^t-1\}|\cr
&=\frac 1l|\{x\in\f_{2^l}:o(x)=d,\ \text{Tr}(x)=0\}|,
\end{align*}
where $o(x)$ denotes the multiplicative order of $x$ and $\text{Tr}=\text{Tr}_{2^l/2}$. Indeed, if $X^l+b_{l-1}X^{l-1}+b_{l-2}X^{l-2}+\cdots+b_0\in\f_2[X]$ is irreducible and $x$ is a zero of this polynomial, then $b_{l-1}=\text{Tr}(x)$ and there are $l$ such zeros. Note that $N_d$ depends only on $l$ but not on $t$. 
An irreducible polynomial in $\f_2[X]$ of the form $X^l+b_{l-2}X^{l-2}+\cdots+b_0$ is said to have {\em zero trace}. Consider an arbitrary factor $f$ of $X^n-1$. For each $d\mid t$, among the irreducible factors of $f$ of degree $o_d(2)$, let $\mu_d$ be the number of those (counting multiplicity) with zero trace and $\nu_d$ be the number of those with nonzero trace. Then $\mu_d\le 2^eN_d$ and $\mu_d+\nu_d\le 2^e\phi(d)/o_2(d)$. Moreover, $f$ is of the form $X^k+a_{k-2}X^{k-2}+\cdots+a_0$ if and only if $\sum_{d\mid t}\nu_d$ is even. Therefore, $X^n-1$ has a factor $X^k+a_{k-2}X^{k-2}+\cdots+a_0\in\f_2[X]$ if and only if
\begin{equation}\label{44.1}
2\le k=\sum_{d\mid t}(\mu_d+\nu_d)o_d(2)
\end{equation}
for some integer sequences $\mu_d$ and $\nu_d$ such that 
\[
\begin{cases}
0\le\mu_d\le 2^eN_d,\vspace{0.2em}\cr
0\le\nu_d\le 2^e(\phi(d)/o_d(2)-N_d),\vspace{0.3em}\cr
\displaystyle\sum_{d\mid t}\nu_d\equiv 0\pmod 2.
\end{cases}
\]
Let $\mathcal K_n$ denote the set of integers $k$ in \eqref{44.1}. Then Corollary~\ref{C2.2} can be stated as

\begin{cor}\label{C4.1}
$\fin$ is not $k$th order sum-free if $k\in\mathcal K_n$.
\end{cor}

The values of $o_d(2)$, $\phi(d)/o_d(2)$ and $N_d$ ($1\le d\le 31$, $d$ odd) are given in Table~\ref{tb1} and the sets $\mathcal K_n$ ($1\le n\le 32$) are given in Table~\ref{tb2}. Note that the examples in Section 2 are covered by Table~\ref{tb2}.

\begin{table}%[ht]
\caption{$o_2(d)$, $\phi(d)/o_d(2)$ and $N_d$ ($1\le d\le 31$, $d$ odd)}\label{tb1}
   \renewcommand*{\arraystretch}{1.4}
    \centering
     \begin{tabular}{c|ccc}
         \hline
         $d$  &  $o_d(2)$ & $\phi(d)/o_d(2)$ &$N_d$ \\ \hline
         1 & 1 & 1 & 0 \cr 

3 & 2 & 1 & 0 \cr 

5 & 4 & 1 & 0 \cr 

7 & 3 & 2 & 1 \cr 

9 & 6 & 1 & 1 \cr 

11 & 10 & 1 & 0 \cr 

13 & 12 & 1 & 0 \cr 

15 & 4 & 2 & 1 \cr 

17 & 8 & 2 & 1 \cr 

19 & 18 & 1 & 0 \cr 

21 & 6 & 2 & 1 \cr 

23 & 11 & 2 & 1 \cr 

25 & 20 & 1 & 1 \cr 

27 & 18 & 1 & 1 \cr 

29 & 28 & 1 & 0 \cr 

31 & 5 & 6 & 3 \cr \hline
     \end{tabular}
\end{table}

\begin{table}%[ht]
\caption{Elements of $\mathcal K_n$ ($1\le n\le 32$)}\label{tb2}
   \renewcommand*{\arraystretch}{1.4}
    \centering
     \begin{tabular}{c|l}
         \hline
         $n$ &\hfil elements of $\mathcal K_n$ \cr \hline
1\cr

2&2\cr

3&3\cr

4&2,4\cr

5&5\cr

6&2,3,4,6\cr

7&3,4,7\cr

8&2,4,6,8\cr

9&3,6,9\cr

10&2,5,8,10\cr

11&11\cr

12&2,3,4,5,6,7,8,9,10,12\cr

13&13\cr

14&2,3,4,5,6,7,8,9,10,11,12,14\cr

15&3,4,5,6,7,8,9,10,11,12,15\cr

16&2,4,6,8,10,12,14,16\cr

17&8,9,17\cr

18&2,3,4,6,8,9,10,12,14,15,16,18\cr

19&19\cr

20&2,4,5,7,8,10,12,13,15,16,18,20\cr

21&3,4,5,6,7,8,9,10,11,12,13,14,15,16,17,18,21\cr

22&2,11,20,22\cr

23&11,12,23\cr

24&2,3,4,5,6,7,8,9,10,11,12,13,14,15,16,17,18,19,20,21,22,24\cr

25&5,20,25\cr

26&2,13,24,26\cr

27&3,6,9,18,21,24,27\cr

28&2,3,4,5,6,7,8,9,10,11,12,13,14,15,16,17,18,19,20,21,22,23,24,25,26,28\cr

29&29\cr

30&2,3,4,5,6,7,8,9,10,11,12,13,14,15,16,17,18,19,20,21,22,23,24,25,26,27,28,30\cr

31&5,6,10,11,15,16,20,21,25,26,31\cr

32&2,4,6,8,10,12,14,16,18,20,22,24,26,28,30,32\cr \hline
     \end{tabular}
\end{table}

\begin{rmk}\label{R4.2}\rm
The number $N_d$ is difficult to compute. Let $l=o_d(2)$. We have
\[
\sum_{\substack{x\in\f_{2^l}\cr o(x)=d}}(-1)^{\text{Tr}(x)}=lN_d-(d-lN_d)=2lN_d-d.
\]
On the other hand, by the M\"obius inversion,
\begin{align*}
\sum_{\substack{x\in\f_{2^l}\cr o(x)=d}}(-1)^{\text{Tr}(x)}\,&=\sum_{d'\mid d}\mu\Bigl(\frac d{d'}\Bigr)\sum_{\substack{x\in\f_{2^l}\cr o(x)\mid d'}}(-1)^{\text{Tr}(x)}\cr
&=\sum_{d'\mid d}\mu\Bigl(\frac d{d'}\Bigr)\frac{d'}{2^l-1}\sum_{y\in\f_{2^l}^*}(-1)^{\text{Tr}(y^{(2^l-1)/d'})}\cr
&=\sum_{d'\mid d}\mu\Bigl(\frac d{d'}\Bigr)\frac{d'}{2^l-1}\Bigl(\sum_{y\in\f_{2^l}}(-1)^{\text{Tr}(y^{(2^l-1)/d'})}-1\Bigr),
\end{align*}
where $\mu(\ )$ is the M\"obius function. Let $\widehat{\f_{2^l}^*}$ denote the group of multiplicative characters of $\f_{2^l}$. In the above
\[
\sum_{y\in\f_{2^l}}(-1)^{\text{Tr}(y^{(2^l-1)/d'})}=\sum_{\substack{\chi\in\widehat{\f_{2^l}^*}\cr o(\chi)\mid(2^l-1)/d'}}G(\chi),
\]
where $G(\chi)$ is the Gauss sum of $\chi$ \cite[Exercise~3.4]{Hou-ams-gsm-2018}. Combining the above equations gives
\[
N_d=\frac1{2l}\biggl(d+\sum_{d'\mid d}\mu\Bigl(\frac d{d'}\Bigr)\frac{d'}{2^l-1}\Bigl(\sum_{\substack{\chi\in\widehat{\f_{2^l}^*}\cr o(\chi)\mid(2^l-1)/d'}}G(\chi)-1\Bigr)\biggr).
\]
Because of the involvement of the Gauss sums, we doubt that $N_d$ can be computed explicitly.
\end{rmk}

%%%%%%%%%%%%%%%%%%%%%%%%%%%%%%%%%%%%%%%%%%%
%  section 4
%%%%%%%%%%%%%%%%%%%%%%%%%%%%%%%%%%%%%%%%%%%
\section{When $k$ Is Small}

The algebraic closure of $\f_q$ is denoted by $\overline\f_q$. A polynomial $f\in\f_q[X_1,\dots,X_k]$ is said to be {\em square-free} if there is no $g\in\f_q[X_1,\dots,X_k]\setminus\f_q$ (equivalently, there is no $g\in\overline\f_q[X_1,\dots,X_k]\setminus\overline\f_q$) such that $g^2\mid f$. It is easy to see that $f$ is square-free if for each $1\le j\le k$, $f$ is a separable polynomial in $X_j$ over $\f_q(X_1,\dots,X_{j-1},X_{j+1},\dots,X_k)$. A polynomial $f$ in $X$ over a field is separable if and only if $\gcd(f,f')=1$. It is clear from \eqref{1.5} that $\Delta(X_1,\dots,X_k)$ is square-free. It follows from \eqref{1.4} that for each $1\le j\le k$, $\Delta_1(X_1,\dots,X_k)$ is a $2$-polynomial in $X_j$ whose coefficient of $X_j$ is nonzero, hence $\Delta_1(X_1,\dots,X_k)$ is separable in $X_j$. Therefore, $\Delta_1(X_1,\dots,X_k)$ is also square-free.

Recall from \eqref{1.5} that $\Delta(\x)=\prod_{\boldsymbol 0\ne\ba\in\f_2^k}\ba\cdot \x$, where $\x=(X_1,\dots,X_k)$. Let $0\le i\le k$. Clearly, $\ba\cdot \x\mid \Delta_i(\x)$ for all $\boldsymbol 0\ne\ba\in\f_2^k$. Hence $\Delta(\x)\mid\Delta_i(\x)$. Let 
\begin{equation}\label{3.0}
F_k(X_1,\dots,X_k)=\frac{\Delta_1(\x)}{\Delta(\x)}\in\f_2[X_1,\dots,X_k].
\end{equation}
We first gather some facts about $F_k$:
\begin{itemize}
\item $F_k$ is homogeneous and symmetric in $X_1,\dots,X_k$.

\item $\deg F_k=2^k-2$, $\deg_{X_i}F_k=2^{k-1}$.

\item $F_k$ is square-free.

\item $F_k\in \f_2[X_2,\dots,X_k][X_1]$ is an affine $2$-polynomial in $X_1$, i.e., the exponents of $X_1$ in $F_k$ are $0, 2^0,\dots,2^{k-1}$.
\end{itemize}

\begin{proof}[Proof of the last claim]
Treat both $\Delta_1(X_1,\dots,X_k)$ and $\Delta(X_1,\dots,X_k)$ as polynomials in $X_1$. Then $\Delta_1(X_1,\dots,X_k)$ is a separable $2$-polynomial with $\deg_{X_1}\Delta_1(X_1,$ $\dots,X_k)=2^k$. Hence its roots (in the algebraic closure of $\f_2(X_2,\dots,X_k)$) form a $k$-dimensional vector space $E$ over $\f_2$. The roots of $\Delta(X_1,\dots,X_k)$ form a $(k-1)$-dimensional subspace of $E$. Therefore, the roots of $F_k(X_1,\dots,X_k)$ form a $(k-1)$-dimensional affine space over $\f_2$, hence the claim.
\end{proof}

For $0\le i<j\le k+1$, define
\[
\Delta_{ij}(X_1,\dots,X_k)=\left|
\begin{matrix}
X_1&\cdots&X_k\cr
\vdots&&\vdots\cr
X_1^{2^{i-1}}&\cdots&X_k^{2^{i-1}}\cr
X_1^{2^{i+1}}&\cdots&X_k^{2^{i+1}}\cr
\vdots&&\vdots\cr
X_1^{2^{j-1}}&\cdots&X_k^{2^{j-1}}\cr
X_1^{2^{j+1}}&\cdots&X_k^{2^{j+1}}\cr
\vdots&&\vdots\cr
X_1^{2^{k+1}}&\cdots&X_k^{2^{k+1}}
\end{matrix}\right|.
\]
From \eqref{3.0}, we have
\begin{equation}\label{4.1}
\Delta_1(X_1,\dots,X_k)=\Delta(X_1,\dots,X_k)F_k(X_1,\dots,X_k).
\end{equation}
Write 
\[
\Delta_1(X_1,\dots,X_k)=A_kX_1^{2^k}+A_{k-1}X_1^{2^{k-1}}+\cdots+A_2X_1^{2^2}+A_0X_1,
\]
where
\[
A_i=\begin{cases}
\Delta(X_2,\dots,X_k)^4&\text{if}\ i=0,\cr
\Delta_{1i}(X_2,\dots,X_k)&\text{if}\ 2\le i\le k,
\end{cases}
\]
and write
\[
\Delta(X_1,\dots,X_k)=B_{k-1}X_1^{2^{k-1}}+B_{k-2}X_1^{2^{k-2}}+\cdots+B_0X_1,
\]
where
\[
B_i=\Delta_i(X_2,\dots,X_k),\quad 0\le i\le k-1.
\]
Further write
\begin{equation}\label{4.2}
F_k(X_1,\dots,X_k)=C_{k-1}X_1^{2^{k-1}}+C_{k-2}X_1^{2^{k-2}}+\cdots+C_0X_1+C_{-1},
\end{equation}
where $C_i\in \f_2[X_2,\dots,X_k]$, $-1\le i\le k-1$. The coefficients $C_i$ ($-1\le i\le k-1$) can be determined in terms of $\Delta_i(X_2,\dots,X_k)$ by comparing the coefficients of $X_1$ in \eqref{4.1}. First, we have $C_{k-1}B_{k-1}=A_k$, whence
\begin{equation}\label{4.3}
C_{k-1}=\frac{A_k}{B_{k-1}}=\frac{\Delta_1(X_2,\dots,X_k)}{\Delta(X_2,\dots,X_k)}=F_{k-1}(X_2,\dots,X_k).
\end{equation}
For $0\le i\le k-2$, we have $C_{k-1}B_i+C_iB_{k-1}=0$, whence
\begin{equation}\label{4.4}
C_i=C_{k-1}\frac{B_i}{B_{k-1}}=F_{k-1}(X_2,\dots,X_k)\frac{\Delta_i(X_2,\dots,X_k)}{\Delta(X_2,\dots,X_k)}.
\end{equation}
Finally, $C_{-1}B_0=A_0$, whence
\begin{equation}\label{4.5}
C_{-1}=\frac{A_0}{B_0}=\frac{\Delta(X_2,\dots,X_k)^4}{\Delta(X_2,\dots,X_k)^2}=\Delta(X_2,\dots,X_k)^2.
\end{equation}

As a by-product, we have a formula for $\Delta_{1i}(X_2,\dots,X_k)$ with $2\le i\le k-1$. (There is no need to consider $\Delta_{1k}(X_2,\dots,X_k)$ since $\Delta_{1k}(X_2,\dots,X_k)=\Delta_1(X_2,\dots,$ $X_k)$.) From \eqref{4.1}, we have
\[
A_i=B_{i-1}C_{i-1}+B_iC_{-1}.
\]
Hence, with $\x'=(X_2,\dots,X_k)$,
\begin{align}
\Delta_{1i}(\x')\,&=\Delta_{i-1}(\x')\frac{\Delta_1(\x')}{\Delta(\x')}\frac{\Delta_{i-1}(\x')}{\Delta(\x')}+\Delta_i(\x')\Delta(\x')^2\\ \label{4.6}
&=\frac{\Delta_1(\x')(\Delta_{i-1}(\x'))^2}{(\Delta(\x'))^2}+\Delta_i(\x')\Delta(\x')^2.\nonumber
\end{align}

The polynomial $F_k(X_1,\dots,X_k)$ contains critical information about the sum-freedom of $\fin$. Recall that $\fin$ is not $k$th order sum-free if and only if there exist $v_1,\dots,v_k\in\f_{2^n}$ such that $\Delta_1(v_1,\dots,v_k)=0$ but $\Delta(v_1,\dots,v_k)\ne 0$. By \eqref{3.0}, this happens if and only if there exist $v_1,\dots,v_k\in\f_{2^n}$ such that $F_k(v_1,\dots,v_k)=0$ but $\Delta(v_1,\dots,v_k)\ne 0$.

A polynomial $f\in\f_q[X_1,\dots,X_k]$ is said to be {\em absolutely irreducible} if it is irreducible in $\fbar_q[X_1,\dots,X_k]$. For $f(X_1,\dots,X_k)\in\f_q[X_1,\dots,X_k]$, define
\[
V_{\f_q^k}(f)=\{(x_1,\dots,x_k)\in\f_q^k:f(x_1,\dots,x_k)=0\}.
\]

\begin{lem}\label{T3.1}
When $k\ge 3$, $F_k(X_1,\dots,X_k)$ is absolutely irreducible.
\end{lem}

\begin{proof}
Since $k\ge 3$, $\deg F_{k-1}(X_2,\dots,X_k)>0$. Since $\Delta_1(X_2,\dots,X_k)$ is square-free, we have by \eqref{4.3} that $\text{gcd}(C_{k-1},\Delta(X_2,\dots,X_k))=1$. Hence by \eqref{4.5},\break 
$\text{gcd}(C_{k-1},C_{-1})=1$. Thus $F_k(X_1,\dots,X_k)$ as a polynomial in $X_1$ over $\fbar_2[X_2,\dots,X_k]$ is primitive (recall that a polynomial over a unique factorization domain such as $\bar {\mathbb F}_2[X_2,...,X_k]$ is said to be primitive if the gcd of its coefficients is 1). Let $f\in\fbar_2[X_2,\dots,X_k]$ be any irreducible factor of $C_{k-1}$. Then $f\mid C_i$ for $0\le i\le k-1$ (by \eqref{4.4}), $f\nmid C_{-1}$, and $f^2\nmid C_{k-1}$. By Eisenstein's criterion \cite[Chapter III, Theorem~6.15]{Hungerford-1980}, $F_k(X_1,\dots,X_k)$ is irreducible in $\fbar_2[X_1,\dots,X_k]$.
\end{proof}

\medskip
\noindent{\bf Remark.} When $k=2$, $F_2(X_1,X_2)$ is not absolutely irreducible. We have
\[
F_2(X_1,X_2)=X_1^2+X_1X_2+X_2^2=(X_1+uX_2)(X_1+(u+1)X_2),
\]
where $u\in\f_{2^2}\setminus\f_2$. Of course, every homogeneous polynomial in two variables over a field $\f$ is a product of linear polynomials over $\fbar$.

\begin{thm}\label{T3.4}
Assume that $k\ge 3$ and
\begin{equation}\label{3.5}
n\ge\frac{13}3k+3.
\end{equation}
Then $\fin$ is not $k$th order sum-free.
\end{thm}

\begin{proof}
Let $q=2^n$. It suffices to show that
\[
V_{\f_q^k}(F_k)\not\subset V_{\f_q^k}(\Delta(X_1,\dots,X_k)).
\]
Since $F_k$ is absolutely irreducible of degree $2^k-2$, by the Lang-Weil bound, as stated in \cite[Theorem~5.2]{Cafure-Matera-FFA-2006}, 
\begin{align*}
|V_{\f_q^k}(F_k)|\,&\ge q^{k-1}-(2^k-3)(2^k-4)q^{k-3/2}-5(2^k-2)^{13/3}q^{k-2}\cr
&>q^{k-1}-2^{2k}q^{k-3/2}-5\cdot 2^{13k/3}q^{k-2}.
\end{align*}
On the other hand, by \cite[Lemma~2.2]{Cafure-Matera-FFA-2006},
\[
|V_{\f_q^k}(F_k)\cap V_{\f_q^k}(\Delta(X_1,\dots,X_k))|\le (2^k-1)^2q^{k-2}<2^{2k}q^{k-2}.
\]
Hence it suffices to show that
\begin{equation}\label{3.6}
q^{k-1}-2^{2k}q^{k-3/2}-5\cdot 2^{13k/3}q^{k-2}\ge 2^{2k}q^{k-2}.
\end{equation}
Let $y=q^{1/2}=2^{n/2}$. Then \eqref{3.6} is equivalent to 
\[
y^2-2^{2k}y-(5\cdot 2^{13k/3}+2^{2k})\ge 0.
\]
Let $y_0$ denote the larger root of the quadratic $Y^2-2^{2k}Y-(5\cdot2^{13k/3}+2^{2k})$. Then
\begin{align*}
y_0\,&=\frac 12\bigl(2^{2k}+\sqrt{2^{4k}+20\cdot 2^{13k/3}+2^{2k+2}}\bigr)\cr
&\le\frac 12\bigl(2^{2k}+\sqrt{21\cdot 2^{13k/3}}\bigr)\cr
&\le \frac 12(1+\sqrt{21})2^{13k/6}.
\end{align*}
Therefore, it suffices to show that
\[
y=2^{n/2}\ge (1+\sqrt{21})2^{13k/6-1},
\]
i.e.,
\[
n\ge\frac{13}3k+2\log_2(1+\sqrt{21})-2,
\]
where $2\log_2(1+\sqrt{21})-2\approx 2.96$. Hence the proof is complete.
\end{proof}

Replacing $k$ with $n-k$ in Theorem~\ref{T3.4} gives
\begin{cor}\label{C3.5}
Assume that $n-k\ge 3$ and
\begin{equation}\label{3.7}
n\le 1.3k-0.9.
\end{equation}
Then $\fin$ is not $(n-k)$th order sum-free.
\end{cor}

%%%%%%%%%%%%%%%%%%%%%%%%%%%%%%%%%%%%%%%%%%%
%  section 5
%%%%%%%%%%%%%%%%%%%%%%%%%%%%%%%%%%%%%%%%%%%
\section{The Case of Even $n$}

The following lemma is from \cite{Carlet-pre}. 
\medskip

\begin{lem}\label{C5.1}
{\it Let $n\geq 6$ and let two integers $l\geq 2$ and $r\geq 2$ be such that $lr<n$. If the inverse function is not $l$th order sum-free nor $r$th order sum-free, then it is not $(l+r)$th order sum-free.}
\end{lem}

%\begin{proof}Let $E=\f_{2^l}\subset \f_{2^n}$. Then $L_E(X)=X^{2^l}+X$. Let $F$ be an $r$-dimensional $\f_2$-subspace of $\f_{2^n}$ such that $\sum_{0\ne v\in F}1/v=0$. Let $F'$ be the $\f_{2^l}$-span of $F$ in $\f_{2^n}$. Then $\dim_{\f_{2^l}}F'\le r<n/l=\dim_{\f_{2^l}}\f_{2^n}$. Hence there exists $0\ne a\in\f_{2^n}$ such that $\text{Tr}_{2^n/2^l}(av)=0$ for all $v\in F'$, equivalently, $\text{Tr}_{2^n/2^l}(av)=0$ for all $v\in F$. It follows that $aF\subset L_E(\f_{2^n})$. Since $L_E:\f_{2^n}\to\f_{2^n}$ induces an isomorphism $\f_{2^n}/E\to L_E(\f_{2^n})$, there exists an $\f_2$-subspace $E'$ of $\f_{2^n}$ of dimension $r$ such that $E'\cap E=\{0\}$ and $L_E(E')=aF$. Now by \cite[Corollary 5]{Carlet-pre}, \[\sum_{0\ne u\in E\oplus E'}\frac 1u=\sum_{0\ne v\in L_E(E')}\frac 1v=\sum_{0\ne v\in F}\frac1{av}=\frac 1a\sum_{0\ne v\in F}\frac 1v=0.\]\end{proof}

We deduce:

\begin{thm}\label{T5.1}
Assume that $n$ is even and $2\le k\le n-2$. Then $\fin$ is not $k$th order sum-free.
\end{thm}

\begin{proof}
It suffices to consider odd $k$ with $3\le k\le n/2$. By \cite[Corollary 4]{Carlet-pre}, $\fin$ is not $3$rd order sum-free. If $3<n/2$, by Lemma~\ref{C5.1}, $\fin$ is not $(2+3)$th order sum-free. If $2+3<n/2$, by Lemma~\ref{C5.1} again, $\fin$ is not $(2+2+3)$th order sum-free. In this way, all odd integers $k$ with $3\le k\le n/2$ are covered.
\end{proof}

%%%%%%%%%%%%%%%%%%%%%%%%%%%%%%%%%%%%%%%%%%%

%%%%%%%%%%%%%%%%%%%%%%%%%%%%%%%%%%%%%%%%%%%

\end{document}